\documentclass[review,onefignum,onetabnum]{siamonline171218}
\usepackage{caption}
\usepackage{subcaption}
\usepackage{cite}

\usepackage{lipsum}
\usepackage{amsfonts}
\usepackage{graphicx}
\usepackage{epstopdf}
\usepackage{algorithmic}
\ifpdf
  \DeclareGraphicsExtensions{.eps,.pdf,.png,.jpg}
\else
  \DeclareGraphicsExtensions{.eps}
\fi

\usepackage{enumitem}
\setlist[enumerate]{leftmargin=.5in}
\setlist[itemize]{leftmargin=.5in}

\newsiamremark{remark}{Remark}
\newsiamremark{hypothesis}{Hypothesis}
\crefname{hypothesis}{Hypothesis}{Hypotheses}
\newsiamthm{claim}{Claim}

\headers{}{C. Kann and M. Feng}

\title{A Repulsive Bounded-Confidence Model of Opinion Dynamics in Polarized Communities\thanks{Submitted to the editors December 8, 2022.
\funding{This work was partially funded by the James S. McDonnell Foundation Postdoctoral Fellowship }}}

\author{Claudia Kann\thanks{Department of Humanities and Social Sciences, California Institute of Technology
  (\email{ckann@caltech.edu}).}
\and Michelle Feng\thanks{Computing + Mathematical Sciences Department, California Institute of Technology 
  (\email{mfeng@caltech.edu}).}}

\usepackage{amsopn}

\makeatletter
\newcommand*{\addFileDependency}[1]{
  \typeout{(#1)}
  \@addtofilelist{#1}
  \IfFileExists{#1}{}{\typeout{No file #1.}}
}
\makeatother

\newcommand*{\myexternaldocument}[1]{%
    \externaldocument{#1}%
    \addFileDependency{#1.tex}%
    \addFileDependency{#1.aux}%
}

\renewcommand{\vec}{\mathbf}

\hypersetup{
  pdftitle={A Repulsive Bounded-Confidence Model of Opinion Dynamics in Polarized Communities},
  pdfauthor={Kann, Claudia and Feng, Michelle}
}

\myexternaldocument{ex_supplement}

\begin{document}

\maketitle

\begin{abstract} Collective opinions affect civic participation, governance, and societal norms. Due to the influence of opinion dynamics, many models of their formation and evolution have been developed. A commonly used approach for the study of opinion dynamics is bounded-confidence models. In these models, individuals are influenced by the opinions of others in their network. They generally assume that individuals will formulate their opinions to resemble those of their peers. In this paper, inspired by the dynamics of partisan politics, we introduce a bounded-confidence model in which individuals may be repelled by the opinions of their peers rather than only attracted to them. We prove convergence properties of our model and perform simulations to study the behavior of our model on various types of random networks. In particular, we observe that including opinion repulsion leads to a higher degree of opinion fragmentation than in standard bounded-confidence models.
\end{abstract}



\section{Introduction}
Opinions dictate how individuals interact with society. They influence who we are friends with, how we vote, and what we consume. At the individual and collective level, opinions shape our lives and our social interactions. Understanding how opinions are formed and their dynamics provides a framework for studying changes in our society. The role of opinions in politics and governance is a prominent part of public discourse in the U.S. Inspired by discussions of political polarization and partisan politics, this paper presents a mathematical approach to modelling polarized opinion dynamics where individuals feel both attractive and repulsive forces.

The influence of public opinion on politics have been studied by philosophers, sociologists, and social theorists\cite{habermas1991,blumer1948,speier1950}.
Contemporary approaches to studying opinions frequently seek to quantify them. In this paper, we focus on the dynamics of opinions. We are interested in studying how opinions in a society shift as a result of relationships between individuals. Various models for studying individual opinions exist\cite{grabowski2006,clifford1973,hk2002,martins2008}. We will focus on bounded-confidence models. Bounded-confidence models are a class of models that suppose individuals change their opinions based on their relationships, when their opinions are already close to those of their peers. That is, if someone's opinion is very far away from my own, even if I have a relationship with them, I will not base my opinions on theirs. Many bounded-confidence models have been developed and studied. They include examinations of consensus formation\cite{dittmer2001,fortunato2005}, polarization\cite{hegselmann2020,sirbu2019}, and a large variety of model extensions for application to real-world opinions \cite{kan2021,hickok2022,brooks2020,altafini2018}. 

We consider polarization, and the notion that individuals may form their opinions by being contrarian. If I have an adversarial relationship with someone, I may specifically choose to hold an opinion that is different from their's. Similar to other bounded-confidence models, we maintain the idea that individuals are mostly influenced by others whose opinions are already somewhat close to our own. We are most interested in understanding how collective opinions in this model behave. What types of relationships and community structures lead to strong polarization within a society? How might we extend those observations to real-world applications and data?

The paper is organized as follows. We introduce the motivation for our model in 
\cref{sec:bg} and define our model in \cref{sec:model}. We present analytical results in \cref{sec:proofs}, and perform numerical simulations on synthetic networks (\cref{sec:synthetic}). Conclusions follow in
\cref{sec:conclusions}.

\section{Background and motivation}
\label{sec:bg}
In this section, we introduce the motivation for our proposed model of opinion dynamics. In \cref{ssec:political_bg}, we discuss political science research which motivates our modelling choices, and in \cref{ssec:bc_bg} we introduce the Hegselmann--Krause model for opinion dynamics, which we use as a starting point in the formulation of our model.

\subsection{Political Science motivation}
\label{ssec:political_bg}
In political science it is common to think of ideologies as points in space, as being on the left or the right, liberal or conservative. This spatial view of politicians and individuals drives much of the work that is done on voting behavior, both at the individual and legislative levels, as well as the models of strategic behavior within Congress. The original conception of this model is often attributed to Downs and his median voter theory \cite{downs1957economic}. This work was followed by further theoretical work on legislative organization \cite{baron1994sequential, riker1980implications, shepsle1979institutional,hitt2017spatial}, electoral competition \cite{ansolabehere2001candidate}, and the courts \cite{mcnollgast1994politics} to name a few. 

The most common method of obtaining ideological spacial estimates for members of congress is NOMINATE \cite{poole1985spatial}. It uses the observed voting choices and an item response model (IRT) to recover spatial distances. This work has been expanded to include bridges over time to estimate changes in the distribution og congressional representatives across congresses \cite{poole2000congress}. More recently, such bridging techniques and new data sources have been used in order to get consistent measurements for politicians in different chambers as well as candidates who do not win their election \cite{bonica2014mapping, clinton2004statistical, bailey2007comparable, shor2010bridge}. 

In this article we present a bounded confidence model in which there are both attractive and repulsive links between members. This is motivated by the idea of varying salience of issues among members of congress. While representatives may have ideological positions that can be uncovered through voting behavior, there is reason to believe that politicians are drawn to fellow representatives with similar priorities. Therefore, working with other members of congress causes their ideologies to converge. In contrast, they make a point of distancing themselves from representatives who's salient issues run in opposition to them, regardless of other similarities. This would cause them to attempt to distinguish themselves. From an electoral perspective, this distinguishing is important and has not yet, to the our knowledge, been accounted for in spatial models.

\subsection{Bounded-Confidence models}
\label{ssec:bc_bg}
The model we propose is a variant of the Hegselmann--Krause (HK) model\cite{hk2002}. The HK model considers the opinions of a group of interacting agents who influence each other. In the HK model, agents are modelled in a network, with connections between them. Agents who are connected to each other will affect each others' opinions, but only if their opinions are sufficiently close. That is, even if two agents are connected, if their opinions are far apart, they will not take each other into consideration as they form new opinions.

The precise mathematical statement of HK is as follows. Suppose $G = (V,E)$ is a network, with associated adjacency matrix $A$. Then at each time step $t$, we denote the opinions of nodes $i \in V$ with the opinion vector $\vec{x}(t)$. We associate to the model a confidence bound $c$. Opinions are updated according to the following rule:
\begin{equation}
    x_i(t+1) = \frac{\sum_{j \in V} A_{ij}x_j(t)\vec{1}_{|x_j(t) - x_i(t)|<c}}{\sum_{j \in V}A_{ij}\vec{1}_{|x_j(t) - x_i(t)|<c}}
    \label{eq:HK1}
\end{equation}
That is, at timestep $t+1$, we examine all neighbors of $i$ which are within the confidence bound, and then average their opinions. Note that we can reformulate this as 
\begin{equation}
    x_i(t+1) = x_i(t) + \frac{\sum_{j \in V}A_{ij}(x_j(t) - x_i(t))\vec{1}_{|x_j(t) - x_i(t)|<c}}{\sum_{j \in V}A_{ij}\vec{1}_{|x_j(t) - x_i(t)|<c}}
    \label{eq:HK2}
\end{equation}

While the model formulation given above is for 1-dimensional opinions, by expanding the notation, the same averaging scheme can be used for higher-dimensional opinions as well.

Previous studies of the HK model have found that it converges in polynomial time\cite{bhattacharyya2013}. In \cite{hk2002}, Hegselmann and Krause also investigated the steady states of the model. Specifically, the HK model suggests that as the confidence bound increases, there is a transition between three types of steady states. For low confidence bounds, the steady state has many possible opinions and no particularly dominant opinions (we refer to this as {\em fragmentation}). As the confidence bound increases, steady states begin to exhibit only a small number of dominant opinions ({\em polarization}). For confidence bounds beyond a certain threshold, we observe only a single dominant opinion ({\em consensus}). These three different steady states can be seen in \cref{fig:hk_example}. In later sections, we will discuss how the steady states of our model compare.

\begin{figure}[htbp!]
    \centering
    \includegraphics{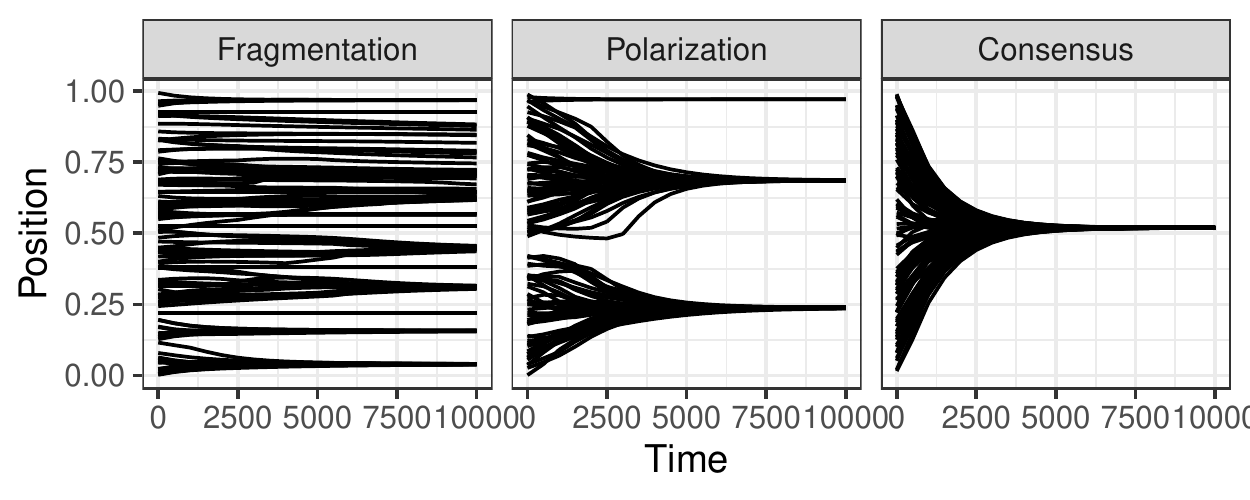}
    \caption{In this figure a Erd\H{o}s--Renyi Random Graph is created with connection probability of 25\%, the evolution three confidence intervals (0.05, 0.2, 0.6) are shown in order to demonstrate the three steady states of the model}
    \label{fig:hk_example}
\end{figure}

\section{Model statement}
\label{sec:model}
The key mechanism that drives our model is the inclusion of repulsive edges, so that individual nodes can push each other way. Suppose $G=(V,E)$ is a network with associated adjacency matrix $A$. For any pair of nodes $i,j$, if $A_{ij} = 1$, there is an attractive edge between them. If $A_{ij} = -1$, there is a repulsive edge. Otherwise, $A_{ij}=0$ and there is no edge between the nodes. As in the original model, we assume that for all $i$, $A_{ii} = 1$. This assumption ensures that a node's existing position is included in the calculation for the next step. We let $\vec{x}$ be the vector of opinions, with $x_i(t)$ representing the opinion of node $i$ at time $t$. 

We define the variable $M$ as follows:
\begin{equation}
    M_{ij}(t) = \left\{
        \begin{array}{c c}
             x_j(t) - x_i(t) & A_{ij} \geq 0 \\
             \text{sign}(x_j(t) - x_i(t))|c - |x_j(t) - x_i(t)||& A_{ij} = -1, |x_j(t) - x_i(t)| > 0 \\
             \text{sign}(j-i)c & A_{ij}= -1, x_j(t) = x_i(t)
        \end{array}
    \right.
\label{eq:distancescaling}
\end{equation}
Intuitively, $M_{ij}$ represents a signed distance which node $i$ will potentially travel because of node $j$. The effect of $M$ is that repulsive forces grow weaker as nodes move farther away from each other. Note that the third row of $M_{ij}$ covers the case where two nodes have the same opinion and repulse each other. In this case, the node with the higher index is pushed towards a higher opinion, while the node with the lower index is pushed towards a lower opinion. In simulation, this situation is unlikely, as it is rare that two nodes which are repulsed share the precise same value. We updated opinions using the following rule:

\begin{equation}
    x_i(t+1) = x_i(t) + \frac{\sum_{j \in V} A_{ij}M_{ij}(t)\mathbf{1}_{|x_j(t) - x_i(t)|<c}}{\sum_{j \in V}|A_{ij}|\mathbf{1}_{|x_j(t)-x_i(t)|<c}}
\label{eq:update}
\end{equation}
Note that if there are no repulsive edges, \cref{eq:update} reduces precisely to the HK model as stated in \cref{eq:HK2}. 

The incorporation of \cref{eq:distancescaling} into the model is to aid convergence. To see why, suppose that we instead naively replaced $M_{ij}(t)$ in \cref{eq:update} with $x_j(t) - x_i(t)$. We can quickly see from the following three node example in \cref{fig:threenode_broken} that opinions might ping pong forever, with attractions pulling opinions together which then rocket away from each other. We further discuss the steady states of this model in~\cref{sec:proofs} and~\cref{sec:synthetic}.

\begin{figure}[!htbp]
    \centering
    \includegraphics[width = \textwidth]{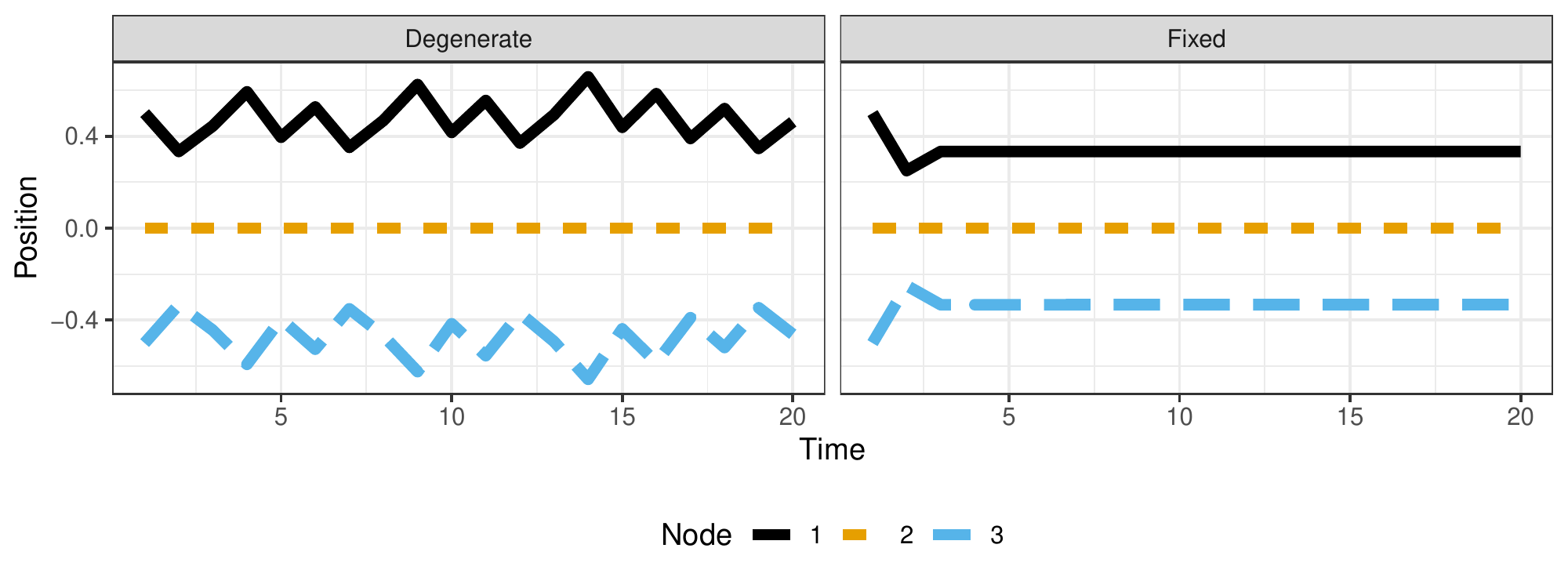}
    \caption{Example with three nodes, where the top and bottom node repulse each other, while the central node attracts both of the others. Without the incorporation of distance scaling, the central node pulls the two outer nodes towards it until they are within confidence of each other, at which point they push each other away until they are no longer within confidence of each other. However, because they are still within confidence of the central node, they are pulled back in, and the cycle repeats so that the model never converges. Using the value of $M_{ij}(t)$ represented in \cref{eq:distancescaling}.}
    \label{fig:threenode_broken}
\end{figure}

When there exist repulsive edges, our model gives rise to several forms of behaviors that differ from the standard HK model. First, the initial range of opinions does not necessarily bound the final set of opinions. In our model, if there exist enough repulsive edges, it is possible for the final opinions to span a much wider range than the initial opinions (as shown in \cref{fig:widthexample}). We prove a bound on final opinions in \cref{thm:width}. 

\begin{figure}[htbp!]
    \centering
    \includegraphics[width = \textwidth]{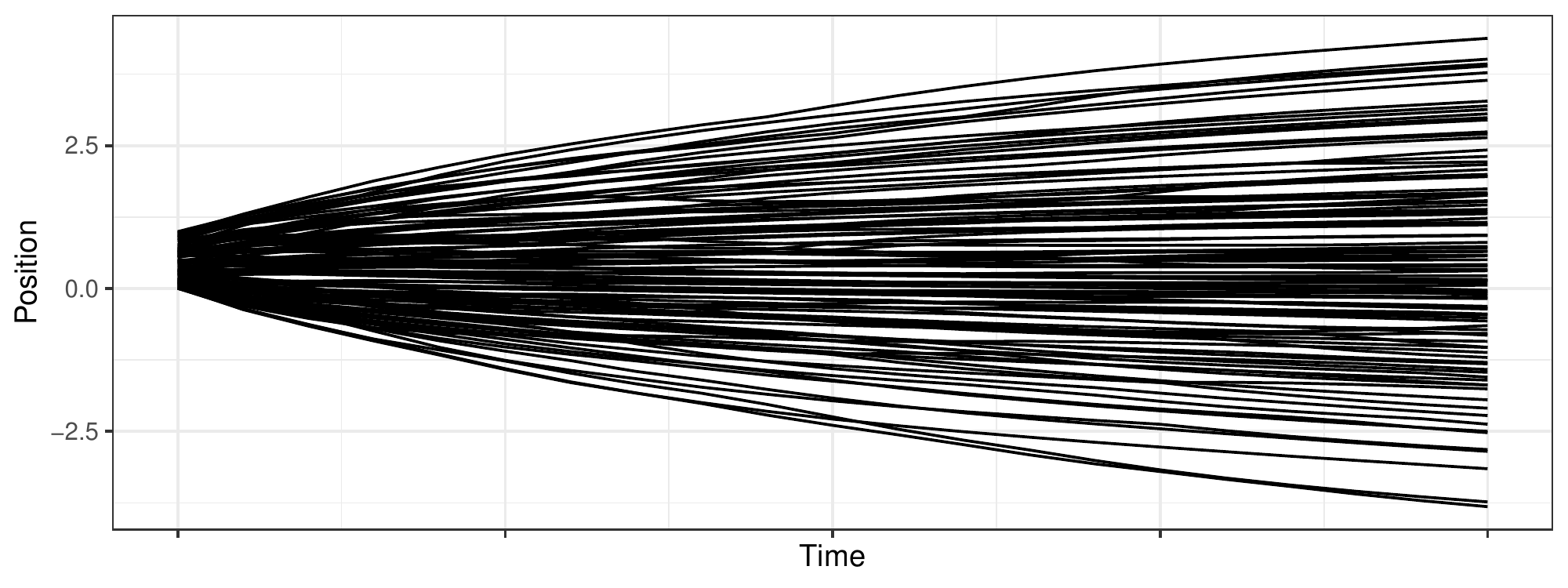}
    \caption{This image shows a simulation where final opinion width was wider than initial opinion width. The edges are created as described in Equation \cref{eq:ER} where $p_1 = 40\%$, $p_2 = 80\%$ and a confidence bound of 1.6}
    \label{fig:widthexample}
\end{figure}

Second, with repulsive edges, connected nodes within confidence of each other may not converge to a single opinion. In HK, we can consider the {\em receptivity subgraph}, or the subgraph of $G$ where edges are pruned if they connect nodes outside confidence bound of each other. In HK, the connected components of the receptivity subgraph will converge to single opinions. In our model, because tensions between attractive and repulsive edges exist, it is possible for nodes to converge to an opinion which is different from its neighbors at stopping time. For example, the same three-node example in \cref{fig:threenode_broken} converges to a state where all three nodes are still connected and within confidence of each other, but do not have the same opinion.

\section{Analytical results}
\label{sec:proofs}

In this section we give some simple proofs about steady states of our model.
We note that while the model converges in most cases (we suspect all, based on the analytical results in \cref{sec:synthetic}), it is not necessarily true that the final opinions are bounded by the initial opinions. The inclusion of negative edges means that repulsive forces between nodes can push the final opinions well outside the bounds of the initial opinions, though this process can only continue so far before nodes are no longer within confidence of each other. To that end, we propose a simple bound on final opinions based on the number of negative edges in~\cref{thm:width}.

We first prove the bound in the simplest case (that of 2 nodes).
\begin{proposition}
Let $G = (V,E)$ be a network with 2 nodes , so that $V = \{0, 1\}$. Then the dynamical process described by~\cref{eq:update} converges, and at time of convergence $T$, 
\begin{equation*}
    \left|x_0(T) - x_1(T) \right| \leq \max\left\{c, \left|x_0(0) - x_1(0)\right|\right\}
\end{equation*}
\end{proposition}
\begin{proof}
Suppose that there are no edges, or $E = \emptyset$. Then the model converges at time $T=1$, and 
\begin{equation*}
    \left|x_0(T) - x_1(T) \right| = \left|x_0(0) - x_1(0)\right|
\end{equation*}

Now suppose that $E = \{(x_0, x_1)\}$, and $|x_0(0) - x_1(0)| \geq c$. Then the update rule will result in no changes, the model converges at $T=1$
\begin{equation*}
    \left|x_0(T) - x_1(T) \right| = \left|x_0(0) - x_1(0)\right|
\end{equation*}

Now suppose that $|x_0(0) - x_1(0)| < c$. If $A_{01} = -1$, the two nodes repel each other. Then
\begin{align*}
    x_1(1)  &= x_1(0) + \frac{(x_1(0) - x_1(0)) + c - (x_1(0) - x_0(1))}{2} \\
    x_0(1) & = x_0(0) + \frac{(x_0(0) - x_0(0))-( c - (x_1(0) - x_0(0))}{2} \\
    x_1(1) -  x_0(1)&= x_1(0) - x_0(0) + \frac{2(c - (x_1(0) - x_0(0)))}{2} \\
    &= c
\end{align*}
and $x_1(1) - x_0(1) >= c$, so that after this time, these two nodes will no longer affect each other, and cannot push each other further, so the model has converged, and $\max_{i,j}\left|x_0(T) - x_1(T)\right| \leq c$.

If $A_{01}=1$, the two nodes attract each other, and the model is equivalent to standard Hegselmann--Krause, so that we have convergence to a single point and
\begin{equation*}
    |x_0(T) - x_1(T)| = 0
\end{equation*}

This covers all possible cases, and the proposition is proven.
\end{proof}
 
The main point to note from this two-node proof is that the repulsive forces between any two nodes will contribute to attempting to push them apart to a distance of precisely $c$. Also note that any node cannot move more than $c$ in any direction over the course of one timestep, because $|M_{ij}(t)| \leq c$.

We now prove several lemmas that we will use to prove ~\cref{thm:width}.

\begin{lemma}
Suppose $i \in V$ a node. Define the following sets:
\begin{align*}
    V_i^+(t) &= \left\{j \in V : A_{ij} = 1 \text{ and } |x_j(t) - x_i(t)| < c\right\}\\
    U_i(t) &= \left\{j \in V: A_{ij} = -1\text{ and }\left[( 0 < x_j(t) - x_i(t) < c) \text{ or } \left(x_j(t) = x_i(t) \text{ and } j > i\right)\right]\right\} \\
    L_i(t) &= \left\{j \in V: A_{ij} = -1 \text{ and } \left[( 0 < x_i(t) - x_j(t) < c) \text{ or } \left(x_j(t) = x_i(t) \text{ and } i > j\right)\right]\right \}
\end{align*}
Then
\begin{equation}
    x_i(t+1) = \frac{\sum_{j \in V_i^+(t)}x_j(t) + \sum_{j \in U_i(t)} (x_j(t) - c) + \sum_{j \in L_i}(x_j(t)+c)}{|V_i^+(t)| + |U_i(t) + |L_i(t)|}
\label{eq:average}
\end{equation}
\label{lemma:average}
\end{lemma}
\begin{proof}
From~\cref{eq:update}, we rearrange
{\small%
\begin{alignat*}{4}
    x_i(t+1) &= x_i(t) &&+ \frac{\sum_{j \in V} A_{ij}M_{ij}(t)\mathbf{1}_{|x_j(t) - x_i(t)|<c}}{\sum_{j \in V}|A_{ij}|\mathbf{1}_{|x_j(t)-x_i(t)|<c}}\\
    &= x_i(t) &&+\frac{\sum_{j \in V_i^+(t) \cup U_i(t) \cup L_i(t)} A_{ij}M_{ij}(t)}{\sum_{j \in V_i^+(t) \cup U_i(t) \cup L_i(t)}|A_{ij}|}\\
    &= x_i(t) &&+ \frac{\sum_{j \in V_i^+(t)}(x_j(t) - x_i(t))}{|V_i^+(t)| +|U_i(t)| + |L_i(t)|} \\
    &&&+ \frac{\sum_{j \in U_i(t)} (-1)(1)(c - (x_j(t) - x_i(t)))}{|V_i^+(t)| +|U_i(t)| + |L_i(t)|} \\
    &&&+ \frac{\sum_{j \in L_i(t)}(-1)(-1)(c - (x_i(t) - x_j(t)))}{|V_i^+(t)| +|U_i(t)| + |L_i(t)|} \\
    &= \frac{\sum_{j \in V_i^+(t)}x_j(t) + \sum_{j \in U_i(t)} (x_j(t) - c) + \sum_{j \in L_i}(x_j(t)+c)}{|V_i^+(t)| + |U_i(t) + |L_i(t)|}\span\span\span
\end{alignat*}
}
\end{proof}

Intuitively, this lemma tells us that the update rule moves $x_i(t)$ to $x_i(t+1)$ by taking an average of several opinions. The set $V_i^+(t)$ contains nodes $i$ is attracted to at time $t$. The set $U_i(t)$ contains nodes which repulse $i$ at time $t$, and which will push $i$'s opinion lower. The set $L_i(t)$ contains nodes which repulse $i$ at time $t$, and which will push $i$'s opinion higher.~\Cref{eq:average} tells us that we can take the average of $x_j(t)$ for $j \in V_i^+(t)$, $x_j(t) - c$ for $j \in U_i(t)$, and $x_j(t) + c$ for $j \in L_i(t)$ to determine $x_i(t+1)$. 

\begin{lemma}Let $i\in V$ at time $t$, and let $W(t) \subset V$ be a set of nodes such that $W(t)$ is completely contained in $V_i^+(t) \cup U_i(t)\cup L_i(t)$. Define
\begin{equation*}
    \overline{W}(t) = \frac{\sum_{j \in W(t)} x_j(t)}{|W(t)|}
\end{equation*}
to be the average of $x_j(t)$ for all $j \in W(t)$. Then we can rewrite~\cref{eq:average} as
\begin{equation*}
    x_i(t+1) = \frac{\sum_{j \in \left(V_i^+(t) \cup U_i(t) \cup L_i(t)\right)\setminus W(t)} x_j(t) + \sum_{j \in W(t)}\overline{W}(t) + \left(|L_i(t)| - |U_i(t)|\right)c}{|V_i^+(t)| + |U_i(t) + |L_i(t)|}
\end{equation*}
\label{lemma:grouping}
\end{lemma}
\begin{proof}
We rearrange~\cref{eq:average} as follows:
\begin{align*}
    x_i(t+1)&= \frac{\sum_{j \in V_i^+(t)}x_j(t) + \sum_{j \in U_i(t)} (x_j(t) - c) + \sum_{j \in L_i}(x_j(t)+c)}{|V_i^+(t)| + |U_i(t) + |L_i(t)|} \\
    &=\frac{\sum_{j \in \left(V_i^+(t) \cup U_i(t) \cup L_i(t)\right)} x_j(t) + \left(|L_i(t)| - |U_i(t)|\right)c}{|V_i^+(t)| + |U_i(t) + |L_i(t)|} \\
    &=\frac{\sum_{j \in \left(V_i^+(t) \cup U_i(t) \cup L_i(t)\right)\setminus W(t)} x_j(t) + \sum_{j \in W(t)}\overline{W}(t) + \left(|L_i(t)| - |U_i(t)|\right)c}{|V_i^+(t)| + |U_i(t) + |L_i(t)|}
\end{align*}
\end{proof}
This lemma allows us to replace a group of opinion values of individual nodes with the average of opinion values across the group, in certain situations.

\begin{lemma}
Let $G = (V,E)$ be a network with $n$ nodes and $m$ edges with confidence bound $c$. Suppose that every edge in $G$ is repulsive. At time $t$, suppose $x_i(t) > x_j(t)$ for all other nodes $j$, so that $i$ is the node with the highest opinion value at time $t$. Then $x_i(t+1) > x_j(t+1)$ for all $j$.
\label{lemma:order}
\end{lemma}
\begin{proof}
Note that $V_k^+(t) = \{x_k\}$ for all $k,t$, since every edge in $G$ is repulsive. For convenience we define the following sets:

\begin{align*}
    U_{ij}(t) &= U_j(t)\bigcap U_i(t) \\
    U_{ij}'(t) &= U_i(t) \setminus U_{ij}(t)\\
    L_{ij}(t) &= L_i(t)\bigcap L_j(t) \\
    L_{ji}'(t) &= L_j(t) \setminus L_{ij}(t)\\
    W_{ij}(t) &= L_i(t) \bigcap U_j(t)
\end{align*}
Unpacking this notation, $U_{ij}(t)$ consists of all nodes that repel both $i$ and $j$ downward, while $L_{ij}(t)$ consists of all nodes that repel both $i$ and $j$ upward. $U_{ij}'(t)$ consists of nodes which repel $i$ downward, but not $j$ (note that if $x_i(t) < x_j(t)$, this is automatically empty), while $L_{ji}'(t)$ consists of nodes which repel $j$ upward, but not $i$ (again, if $x_i(t) < x_j(t)$, this is empty). Finally, $W_{ij}(t)$ consists of nodes which repel $i$ upward and $j$ downward (empty if $x_j(t) > x_i(t)$).

Now, suppose $x_i(t) > x_j(t)$ for all $j \in V$. Then we can write
\begin{align*}
    V_i^+(t) &= \{i\} \\
    U_i(t) &= \emptyset \\
    L_i(t) &= L_{ij}(t) \bigcup W_{ij}(t) \bigcup \{j\}
\end{align*}

and
\begin{align*}
    V_j^+(t) &= \{j\} \\
    U_i(t) &= \{i\} \bigcup W_{ij}(t) \\
    L_i(t) &= L_{ij}(t) \bigcup L_{ji}'(t)
\end{align*}

Then we observe the following from the knowledge that nodes only effect each other if they are within confidence of each other.
\begin{alignat*}{4}
    x_j(t) + c &> x_i(t)\\
    \overline{L_{ij}}(t) + c &> x_i(t) \\
    \overline{W_{ij}}(t)+c &>  x_i(t) \\
    x_i(t) &> \overline{L_{ji}'} + c \\
\end{alignat*}

Then applying~\cref{lemma:average} and~\cref{lemma:grouping}:
{\small%
\begin{align*}
    x_i(t+1) &= \frac{x_i(t) + (x_j(t) + c) + |L_{ij}(t)|(\overline{L_{ij}}(t) + c) + |W_{ij}(t)|(\overline{W_{ij}}(t)+c)}{2 + |L_{ij}(t)| + |W_{ij}(t)|} \\
    &> \frac{x_i(t) + (x_j(t) + c) + |L_{ij}(t)|(\overline{L_{ij}}(t) + c) + |W_{ij}(t)|(\overline{W_{ij}}(t)+c) + |L_{ji}'(t)|(\overline{L_{ji}'}(t) + c)}{2 + |L_{ij}(t)| + |W_{ij}(t)| + |L_{ji}'(t)|} \stepcounter{equation}\tag{\theequation}\label{eq:inequality}\\
    &> \frac{(x_i(t)-c) + x_j(t) + |L_{ij}(t)|(\overline{L_{ij}}(t) + c) + |W_{ij}(t)|(\overline{W_{ij}}(t)-c) + |L_{ji}'(t)|(\overline{L_{ji}'}(t) + c)}{2 + |L_{ij}(t)| + |W_{ij}(t)| + |L_{ji}'(t)|} \\
    &= x_j(t+1)
\end{align*}
}
where the inequality in~\cref{eq:inequality} follows because $x_i(t+1)$ is a weighted average, and $\overline{L_{ji}'}(t)$ is less than all of the other values being averaged in the previous line. The next inequality follows straightforwardly by replacing each value in the average with a smaller or equal value. 

So if $x_i(t)$ has the highest value opinion at time $t$, it will always have the highest value opinion.
\end{proof}

\begin{corollary}
Let $G = (V,E)$ be a network with $n$ nodes and $m$ edges with confidence bound $c$. Suppose that every edge in $G$ is repulsive. At time $t$, let $M = \{i : x_i(t) \geq x_j(t) \forall j \in V\}$. Then $x_{\max_M i}(t+1) > x_j(t+1) \forall j \in V$.
\label{cor:order}
\end{corollary}
\begin{proof}
From the definitions of $U_i(t), L_i(t)$, we can observe that the member of $M$ with highest index will have the largest corresponding set $L_i(t)$ and the smallest corresponding $U_i(t)$, so that at time $t+1$, that member of $M$ will have the highest-valued opinion of all members of $M$. By the same logic as in the proof of~\cref{lemma:order}, that opinion will also be the highest-valued opinion overall.
\end{proof}

\begin{corollary}
Let $G = (V,E)$ be the complete network with $n$ nodes with confidence bound $c$. Suppose that every edge in $G$ is repulsive. At time $t$, let $M = \{i : x_i(t) \leq x_j(t) \forall j \in V\}$. Then $x_{\min_M i}(t+1) < x_j(t+1) \forall j \in V$.
\label{cor:lowest}
\end{corollary}
\begin{proof}
Proves that $x_i(t)<x_j(t)$ for all $j \in V$, then $x_i(t+1) < x_j(t+1)$ for all $j \in V$, by segmenting $U_i(t), L_i(t), U_j(t), L_j(t)$ into the appropriate subsets and reversing inequalities as needed as in~\cref{lemma:order}. Then the same logic as in~\cref{cor:order} proves the statement.
\end{proof}

\begin{lemma}
Let $G = (V,E)$ be a network with $n$ nodes and $m$ edges with confidence bound $c$. Suppose that every edge in $G$ is repulsive. At time $t$, suppose $x_i(t) > x_j(t)$ for all other nodes $j \in V$, so that $i$ is the node with the highest-valued opinion at time $t$. Suppose that there is some node $j$ such that $x_i(t) - x_j(t) < c$, and that $j$ has the highest-valued opinion of all such nodes. Then
\begin{equation*}
    \frac{2c}{2 + |L_{ij}(t)| + |L_{ji}'(t)|} \leq x_i(t+1)-x_j(t+1) \leq \frac{(|L_{ji}'(t)| + 2)c}{2 + |L_{ij}(t)| + |L_{ji}'(t)|}
\end{equation*}
\label{lemma:max}
\end{lemma}
\begin{proof}
By assumption, since $j$ has the highest-valued opinion of all nodes within confidence of $i$, $W_{ij}(t)$ is empty. To prove the lower bound,
{\small%
\begin{align*}
    x_i(t+1)  &= \frac{x_i(t) + (x_j(t) + c) + |L_{ij}(t)|(\overline{L_{ij}}(t) + c) }{2 + |L_{ij}(t)| } \\
    &\geq \frac{x_i(t) + (x_j(t) + c) + |L_{ij}(t)|(\overline{L_{ij}}(t) + c)  + |L_{ji}'(t)|(\overline{L_{ji}'}(t) + c)}{2 + |L_{ij}(t)|  + |L_{ji}'(t)|}\\
    x_j(t+1) &= \frac{(x_i(t) -c) + x_j(t)+ |L_{ij}(t)|(\overline{L_{ij}}(t) + c) +  |L_{ji}'(t)|(\overline{L_{ji}'}(t) + c)}{2 + |L_{ij}(t)| + |L_{ji}'(t)|} \\
    x_i(t+1) - x_j(t+1) &\geq  \frac{2c}{2 + |L_{ij}(t)| + |L_{ji}'(t)|}
\end{align*}
}
To prove the upper bound,
\begin{align*}
    x_i(t+1)  &= \frac{x_i(t) + (x_j(t) + c) + |L_{ij}(t)|(\overline{L_{ij}}(t) + c) + }{2 + |L_{ij}(t)| } \\
    &\leq \frac{x_i(t) + (x_j(t) + c) + |L_{ij}(t)|(\overline{L_{ij}}(t) + c)  + |L_{ji}'(t)|(x_j(t) + c)}{2 + |L_{ij}(t)|  + |L_{ji}'(t)|} \\
    x_j(t+1)&= \frac{(x_i(t) -c) + x_j(t)+ |L_{ij}(t)|(\overline{L_{ij}}(t) + c) +  |L_{ji}'(t)|(\overline{L_{ji}'}(t) + c)}{2 + |L_{ij}(t)| + |L_{ji}'(t)|} \\
    x_i(t+1) - x_j(t+1)&\leq \frac{c + c  + |L_{ji}'(t)|(x_j(t) - \overline{L_{ji}'}(t))}{2 + |L_{ij}(t)| + |L_{ji}'(t)|}  \\
    &\leq \frac{(|L_{ji}'(t)| + 2)c}{2 + |L_{ij}(t)| + |L_{ji}'(t)|}
\end{align*}

Notice that if $L_{ji}'(t)$ is empty, both inequalities become equalities, so that
\begin{equation*}
    x_i(t+1) - x_j(t+1) = \frac{2c}{2 + |L_{ij}(t)|}
\end{equation*}
Notice also that if both $L_{ij}(t)$ and $L_{ji}'(t)$ are empty, that the distance between $x_i(t+1)-x_j(t+1)$ is precisely $c$.
\end{proof}

\begin{corollary}
Let $G = (V,E)$ be a network with $n$ nodes and $m$ edges with confidence bound $c$. Suppose that every edge in $G$ is repulsive. At time $t$, suppose $x_i(t) < x_j(t)$ for all other nodes $j \in V$, so that $i$ is the node with the lowest-valued opinion at time $t$. Suppose that there is some node $j$ such that $x_j(t) - x_i(t) < c$, and that $j$ has the lowest-valued opinion of all such nodes. Then
\begin{equation*}
    \frac{2c}{2 + |U_{ij}(t)| + |U_{ij}'(t)|} \leq x_i(t+1)-x_j(t+1) \leq \frac{(|U_{ij}'(t)| + 2)c}{2 + |U_{ij}(t)| + |U_{ij}'(t)|}
\end{equation*}
\label{cor:min}
\end{corollary}

\Cref{lemma:max} and~\cref{cor:min} are interesting because they give us precise conditions under which the nodes with the most extreme opinions will no longer be within confidence of any other nodes. Specifically, in order for the node with the highest-value opinion to lose connection with all other nodes, it must be true that the only node it is still influenced by is the node with the second-highest-value opinion, and that neither of the two nodes is influenced by any other nodes. Otherwise, they will remain within confidence of each other, even as the node with highest-value opinion remains the most extreme node and continues to have its opinion pushed upward.

We conclude with one more lemma about the bound on the width of the gap between consecutive nodes.

\begin{lemma}Let $G = (V,E)$ be the complete network with $n$ nodes and confidence bound $c$. Suppose that every edge in $G$ is repulsive. At time $t$, suppose that $i$ and $j$ are nodes such that $(i,j) \in E$, $x_i(t) > x_j(t)$, and $x_i(t) - x_j(t) < c$, and there exist no nodes $k$ connected to $i$ or $j$ such that
$x_i(t) > x_k(t) > x_j(t)$. Then 
\begin{equation*}
    |x_i(t+1) - x_j(t+1)| \leq c
\end{equation*}
\label{lemma:gapwidth}
\end{lemma}
\begin{proof}
By the assumption that no nodes have values between $x_i(t)$ and $x_j(t)$, we have that $W_{ij}(t) = W_{ji}(t) = \emptyset$. Then to prove one direction of the bound,
{\footnotesize%
\begin{align*}
    x_i(t+1) &= \frac{|U_{ij}'(t)|(\overline{U_{ij}'}(t) - c) + |U_{ij}(t)|(\overline{U_{ij}}(t)-c) + x_i(t) + (x_j(t)+c) + |L_{ij}(t)|(\overline{L_{ij}}(t)+c)}{2 + |U_{ij}'(t)| + |U_{ij}(t)| + |L_{ij}(t)|} \\
    &\leq \frac{|U_{ij}'(t)|(\overline{U_{ij}'}(t) - c) + |U_{ij}(t)|(\overline{U_{ij}}(t)-c) + x_i(t) + (x_j(t)+c) + |L_{ij}(t)|(\overline{L_{ij}}(t)+c) + |L_{ji}'(t)|(x_j(t)+c)}{2 + |U_{ij}'(t)| + |U_{ij}(t)| + |L_{ij}(t)| + |L_{ji}'(t)|} \\
x_j(t+1) &= \frac{|U_{ij}(t)|(\overline{U_{ij}}(t)-c) + (x_i(t)-c) + x_j(t) + |L_{ij}(t)|(\overline{L_{ij}}(t)+c) + |L_{ji}'(t)|(\overline{L_{ji}'}(t)+c)}{2 + |U_{ij}(t)| + |L_{ij}(t)| + |L_{ji}'(t)|} \\
&\geq \frac{|U_{ij}'(t)|(x_i(t)-c) +|U_{ij}(t)|(\overline{U_{ij}}(t)-c) + (x_i(t)-c) + x_j(t) + |L_{ij}(t)|(\overline{L_{ij}}(t)+c) + |L_{ji}'(t)|(\overline{L_{ji}'}(t)+c)}{2 + |U_{ij}'(t)| + |U_{ij}(t)| + |L_{ij}(t)| + |L_{ji}'(t)|} \\
\end{align*}
}
Combining both equations,
\begin{align*}
x_i(t+1) - x_j(t+1) &\leq \frac{|U_{ij}'(t)|\left(\overline{U_{ij}'}(t) - x_i(t)\right) + c + c + |L_{ji}'(t)|\left(x_j(t) - \overline{L_{ji}'}(t)\right)}{2 + |U_{ij}'(t)| + |U_{ij}(t)| + |L_{ij}(t)| + |L_{ji}'(t)|} \\
&\leq \frac{\left(2 + |U_{ij}'(t)| + |L_{ji}'(t)|\right)c}{2 + |U_{ij}'(t)| + |U_{ij}(t)| + |L_{ij}(t)| + |L_{ji}'(t)|} \\
&\leq c
\end{align*}
To prove the other direction,
{\footnotesize%
\begin{align*}
x_j(t+1) &= \frac{|U_{ij}(t)|(\overline{U_{ij}}(t)-c) + (x_i(t)-c) + x_j(t) + |L_{ij}(t)|(\overline{L_{ij}}(t)+c) + |L_{ji}'(t)|(\overline{L_{ji}'}(t)+c)}{2 + |U_{ij}(t)| + |L_{ij}(t)| + |L_{ji}'(t)|} \\
&\leq \frac{|U_{ij}'(t)|(\overline{L_{ij}}(t)+c) +|U_{ij}(t)|(\overline{U_{ij}}(t)-c) + (x_i(t)-c) + x_j(t) + |L_{ij}(t)|(\overline{L_{ij}}(t)+c) + |L_{ji}'(t)|(\overline{L_{ji}'}(t)+c)}{2 + |U_{ij}'(t)| + |U_{ij}(t)| + |L_{ij}(t)| + |L_{ji}'(t)|} \\
x_i(t+1) &= \frac{|U_{ij}'(t)|(\overline{U_{ij}'}(t) - c) + |U_{ij}(t)|(\overline{U_{ij}}(t)-c) + x_i(t) + (x_j(t)+c) + |L_{ij}(t)|(\overline{L_{ij}}(t)+c)}{2 + |U_{ij}'(t)| + |U_{ij}(t)| + |L_{ij}(t)|} \\
&\geq \frac{|U_{ij}'(t)|(\overline{U_{ij}'}(t) - c) + |U_{ij}(t)|(\overline{U_{ij}}(t)-c) + x_i(t) + (x_j(t)+c) + |L_{ij}(t)|(\overline{L_{ij}}(t)+c) + |L_{ji}'(t)|(\overline{U_{ij}}(t)-c)}{2 + |U_{ij}'(t)| + |U_{ij}(t)| + |L_{ij}(t)| + |L_{ji}'(t)|} \\
\end{align*}
}
Combining both inequalities yields
{\footnotesize
\begin{equation*}
    x_j(t+1) - x_i(t+1) \leq \frac{|U_{ij}'(t)|\left(\overline{L_{ij}}(t) - \overline{U_{ij}'}(t) +2c\right) + c + c + |L_{ji}'(t)|\left( \overline{L_{ji}'}(t) - \overline{U_{ij}}(t) + 2c\right)}{2 + |U_{ij}'(t)| + |U_{ij}(t)| + |L_{ij}(t)| + |L_{ji}'(t)|} 
\end{equation*}
}
Note, however, that
\begin{align*}
    2c &= (x_i(t) + c) - (x_i(t) - c) \\
    &> \overline{U_{ij}'}(t) - \overline{L_{ij}}(t) \\
    &> (x_j(t) +c) - x_j(t) = c
\end{align*}
and similarly $c < \overline{U_{ij}}(t) - \overline{L_{ji}'}(t)<2c$ so that we have
{\small%
\begin{align*}
    x_j(t+1) - x_i(t+1) &\leq \frac{|U_{ij}'(t)|\left(\overline{L_{ij}}(t) - \overline{U_{ij}'}(t) +2c\right) + c + c + |L_{ji}'(t)|\left( \overline{L_{ji}'}(t) - \overline{U_{ij}}(t) + 2c\right)}{2 + |U_{ij}'(t)| + |U_{ij}(t)| + |L_{ij}(t)| + |L_{ji}'(t)|} \\
    &\leq \frac{\left(2 + |U_{ij}'(t)| + |L_{ji}'(t)|\right)c}{2 + |U_{ij}'(t)| + |U_{ij}(t)| + |L_{ij}(t)| + |L_{ji}'(t)|} \\
    &\leq c
\end{align*}
}
and the proof is finished.
\end{proof}

\begin{theorem} Let $G = (V, E)$ be a network with $n$ nodes and confidence bound $c$. Suppose that $G$ is the complete graph, and that every edge $(i,j) \in E$ is repulsive (that is $A_{ij} = -1$). Suppose also that we have initial opinions $x_i(0)$ such that $|x_i(0) - x_j(0)| < c$. Then the model converges, and $\max_{i,j}\left|x_i(T) - x_j(T)\right| = (n-1)c$.
\label{thm:width}
\end{theorem}
\begin{proof} 
The intuition for this theorem is as follows: for any repulsive edge $(i,j)$, nodes $i$ and $j$ will repel each other until
\begin{equation*}
    |x_i(t) - x_j(t)| >= c
\end{equation*}
at some future timestep $t$. If every edge is repulsive, we must have distance at least $c$ between every pair of nodes connected by an edge in order for the model to converge. Intuitively, the nodes will always continue to push each other outward until they reach a distance of $c$, and no further, so that the final convergent state of the model will occur when there are gaps of at least $c$ between all of the $m$ edges in the original graph. However, from~\cref{lemma:max}, the gaps will have precisely width $c$, so that the bound holds.

From~\cref{cor:order} and~\cref{cor:lowest}, at time 1, there must be a highest and lowest-value opinion node. By~\cref{lemma:order}, for $t>1$, these nodes will always be the highest and lowest-value opinion nodes. Call these nodes $i_{max}, i_{min}$.

Because $G$ is the complete graph, and all edges are repulsive, we can observe that $i_{max}$ and $i_{min}$ will have their opinions pushed outward, since initially every node is within confidence of every node. Additionally, from~\cref{lemma:grouping}, we can observe that $i_{max}$ will be pushed in the direction of $\overline{\{j \neq i\}}(0) + c$, so that the nodes with opinions much lower valued than the average will start to drop out of confidence of $i_{max}$. Further, from~\cref{lemma:max}, $i_{max}$ will remain within confidence of at least one node as long as it is within confidence of at least 2 nodes in the previous timestep. Combining these lemmas, we can see that eventually at time $t'$, $i_{max}$ will be within confidence of exactly one other node.

Let $i_{max}'$ be the singular node for which $x_{i_{max}}(t') - x_{i_{max}'}(t') < c$. Then we can follow the same proof procedure as in~\cref{lemma:order} to prove that $x_{i_{max}'}(t'+1) > x_j(t'+1)$ for all $j \in V$ other than $j = i_{max}$, and that none of the remaining nodes can be pushed into confidence of $i_max$. We do not include the procedure here because of its similarity to~\cref{lemma:order}, but the key observation that drives the proof is that there is only a single node $i_{max}$ exerting downward pressure on $i_{max}'$ (if a very high number of nodes were exerting downward pressure on $i_{max}'$, it would be possible for $i_{max}'$ to lose its position as the node with second-highest-value opinion). This allows us to rewrite the $x_{i_{max}'}$ as an average of values which preserve the order of $i_{max}, i_{max}',$ and the remaining nodes. Similarly, we can show that there is some time after which the node with the second-lowest-value opinion will always remain the node with the second-lowest-value opinion. 

We continue in this manner, proceeding from the nodes with the highest and lowest-value opinions inwards until we show that after some time, the nodes' opinions must remain in a fixed order. 

From this point on, we observe that from~\cref{lemma:gapwidth}, the gap between any two consecutive nodes is bounded by $c$. Because of our initial conditions on $x_i(t)$, it is impossible for any gap between consecutive nodes to be larger at any point. If any two nodes have a gap smaller than $c$, we will not have converged, as the repulsion between the two nodes will push them apart in the next time step. All nodes will push each other apart until the gap between any two consecutive nodes is precisely $c$, at which point the model has converged. Because there are $n$ nodes, this tells us
\begin{equation*}
    \max\{|x_i(T) - x_j(T)|\} = (n-1)c
\end{equation*}
\end{proof}

The proof for~\cref{thm:width} relies on all edges being repulsive, thereby preserving the ordering of the nodes. This property does not necessarily hold when there are both attractive and repulsive edges. However, we suspect based on numerics that the following theorem is also true:

\begin{theorem}
Suppose $G = (V,E)$ is a network with $n$ nodes, $m$ edges, and confidence bound $c$. Let $m_r$ be the number of repulsive edges in the network. Then the model converges and
\begin{equation*}
    \max_{i,j \in V} \{x_i(T) - x_j(T)\} \leq \max\{\max_{i,j \in V} \{x_i(0) - x_j(0)\}, mc\}
\end{equation*}
\end{theorem}
\begin{proof}[Intuition]
The worst case for this model assumes that all repulsive nodes end up at least $c$ apart from each other, so if all nodes start out within confidence of each other, the worst case is one in which all nodes with repulsive edges are chained together in consecutive order along a line of $m$ edges, in which case the width of their opinions cannot exceed $mc$, since the bounds in~\cref{lemma:gapwidth} should apply and prevent any individual gap from growing wider. The only way a gap could grow wider is if there are attractive nodes pulling the repulsed nodes further apart, in which case those attractive nodes either have repulsive forces between them, and have already been considered in the line, or must have started farther apart to begin with, in which case we look at $\max_{i,j \in V} \{x_i(0) - x_j(0)\}$. 

Because we cannot rely on nodes remaining in fixed order in this case, we cannot use the same technique as in~\cref{thm:width} to prove convergence and a bound. However, in practice, we observe that the range of final opinions increases with the number of repulsive edges, and that in practice the bound of $mc$ is not very tight (this is to be expected, as, for example, in the case of the complete graph in~\cref{thm:width}, the bound is considerably smaller). To see numerics showing that the range of final opinions scales with number of repulsive edges and $c$, see~\cref{fig:er} and associated discussion.
\end{proof}

\section{Numerical results on synthetic networks}
\label{sec:synthetic}
In this section we present analysis of numerical simulations on a variety of random networks, chosen for their usage in modelling social structures\cite{siegel2009}.

\subsection{Erd\H{o}s--Renyi}
\label{sec:ER}
We begin with an adaptation of Erd\H{o}s--Renyi (ER) networks as a simple random network model. To achieve a random network with both positive and negative edges, we generate two ER networks, $G_1 = G(n, p_1)$ and $G_2 = G(n, p_2)$, with associated adjacency matrices $A_1$ and $A_2$. The total network $G$, is then the network derived from the adjacency matrix $A_1 - A_2$. A visual of this method can be seen in \cref{fig:ER_example}. In the subsequent network the probability of each type of edge between any set of nodes can be written as:
\begin{equation}
    P((i,j) \in E) = \left\{
    \begin{array}{cc}
         0 & (1-p_1)(1-p_2) + p_1p_2 \\
         1 & p_1(1-p_2) \\
         -1 & (1-p_1)p_2
    \end{array}
    \right.
    \label{eq:ER}
\end{equation}

\begin{figure}[!ht]
   \centering
     \begin{subfigure}[b]{0.3\textwidth}
         \centering
         \includegraphics[width=\textwidth]{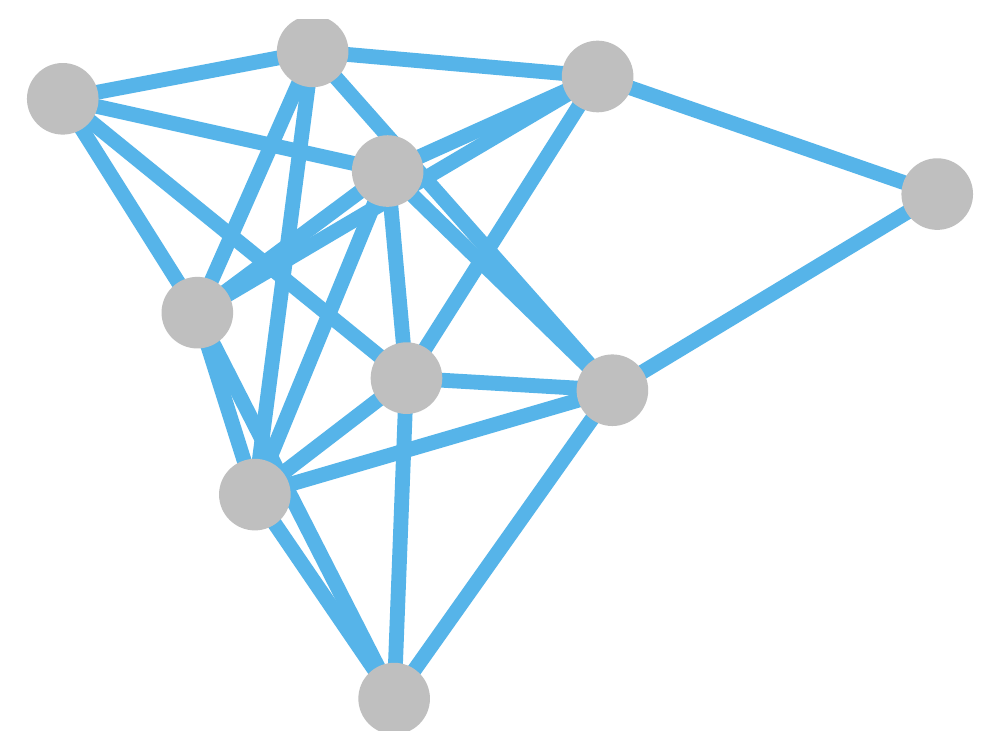}
         \caption{G1: positive node network with $p_1 = 0.6$}
         \label{fig:ER_G1}
     \end{subfigure}
     \hfill
     \begin{subfigure}[b]{0.3\textwidth}
         \centering
         \includegraphics[width=\textwidth]{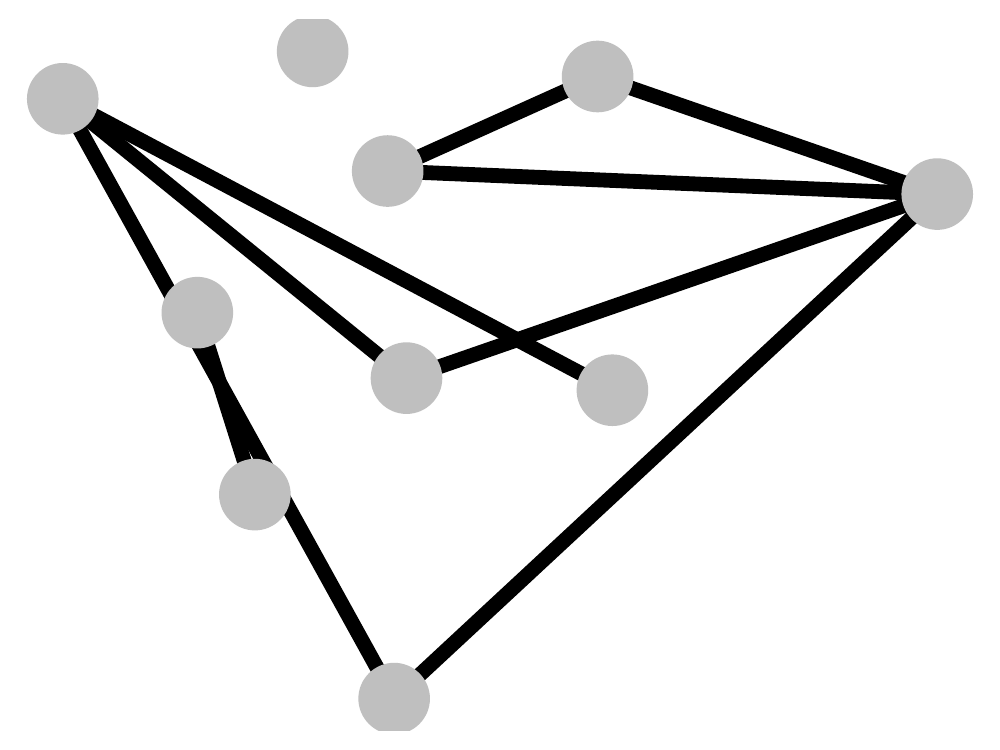}
         \caption{G2: negative node network with $p_2 = 0.2$}
         \label{fig:ER_G2}
     \end{subfigure}
      \hfill
     \begin{subfigure}[b]{0.3\textwidth}
         \centering
         \includegraphics[width=\textwidth]{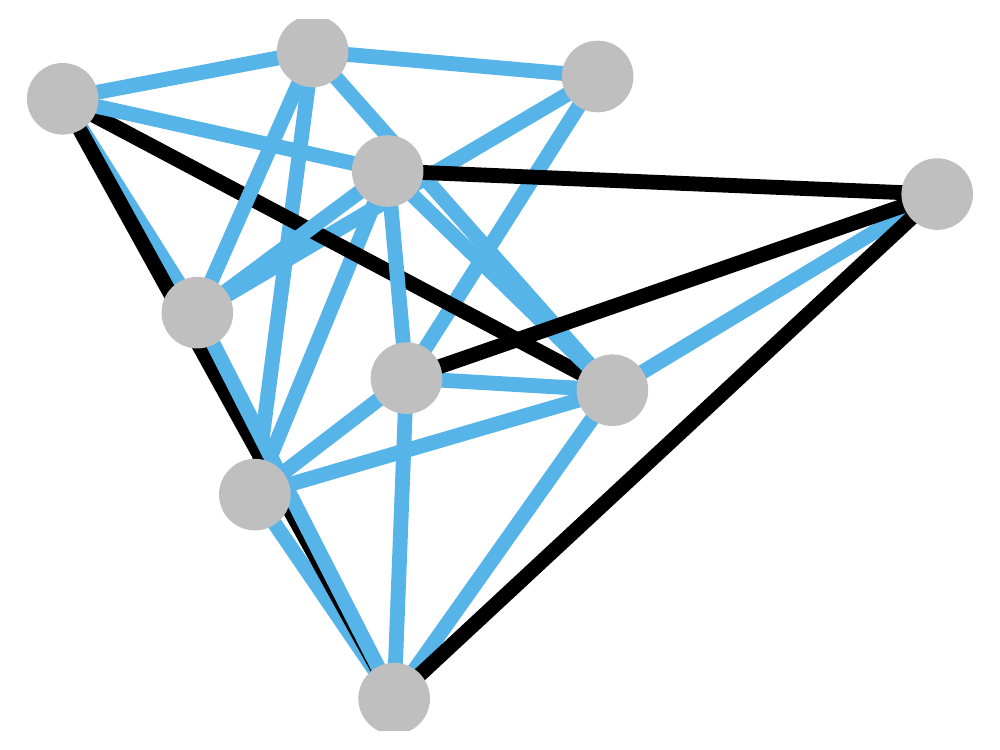}
         \caption{G3: final network with $p_1 = 0.6$ and $p_2 = 0.2$.}
         \label{fig:ER_G3}
     \end{subfigure}
     \caption{An example of the generation of the ER network with attractive and repulsive edges}
     \label{fig:ER_example}
\end{figure}

To create the simulation results, 100 trials were run with all combinations of the following parameters:
\begin{align*}
    p_1 &\in (0.2, 0.4, 0.6, 0.8, 1)\\
    p_2 &\in (0.0, 0.2, 0.4, 0.6, 0.8)\\
    c &\in (0.05, 0.4,  0.8,  1.2, 1.6)
\end{align*}

For each trial, a random set of initial opinions is generated and the model is applied for 10000 iterations. In \cref{fig:er}, one trial is shown for each set of parameters when $p_1$ is set to 0.4. This trial was chosen randomly and all other trials qualitatively look the same. 

The final range of opinions gets wider with both $p_2$ and $c$ once repulsive edges are included. These results are in line with expectations. As $p_2$ increases, so does the number of negative connections, resulting in more repulsive forces between nodes, pushing opinions apart. As $c$ increases, nodes have more neighbors. For $p_2 << p_1$, the attractive forces overpower the repulsive ones, so that higher $c$ leads to more consensus, as in standard HK models. For $p_2 >> p_1$, the opposite is true -- repulsive forces overpower attractive ones, and nodes push each other further apart for higher $c$, resulting in a wider spread of opinions.

\begin{figure}[htbp!]
    \centering
    \includegraphics[width = 0.8\textwidth]{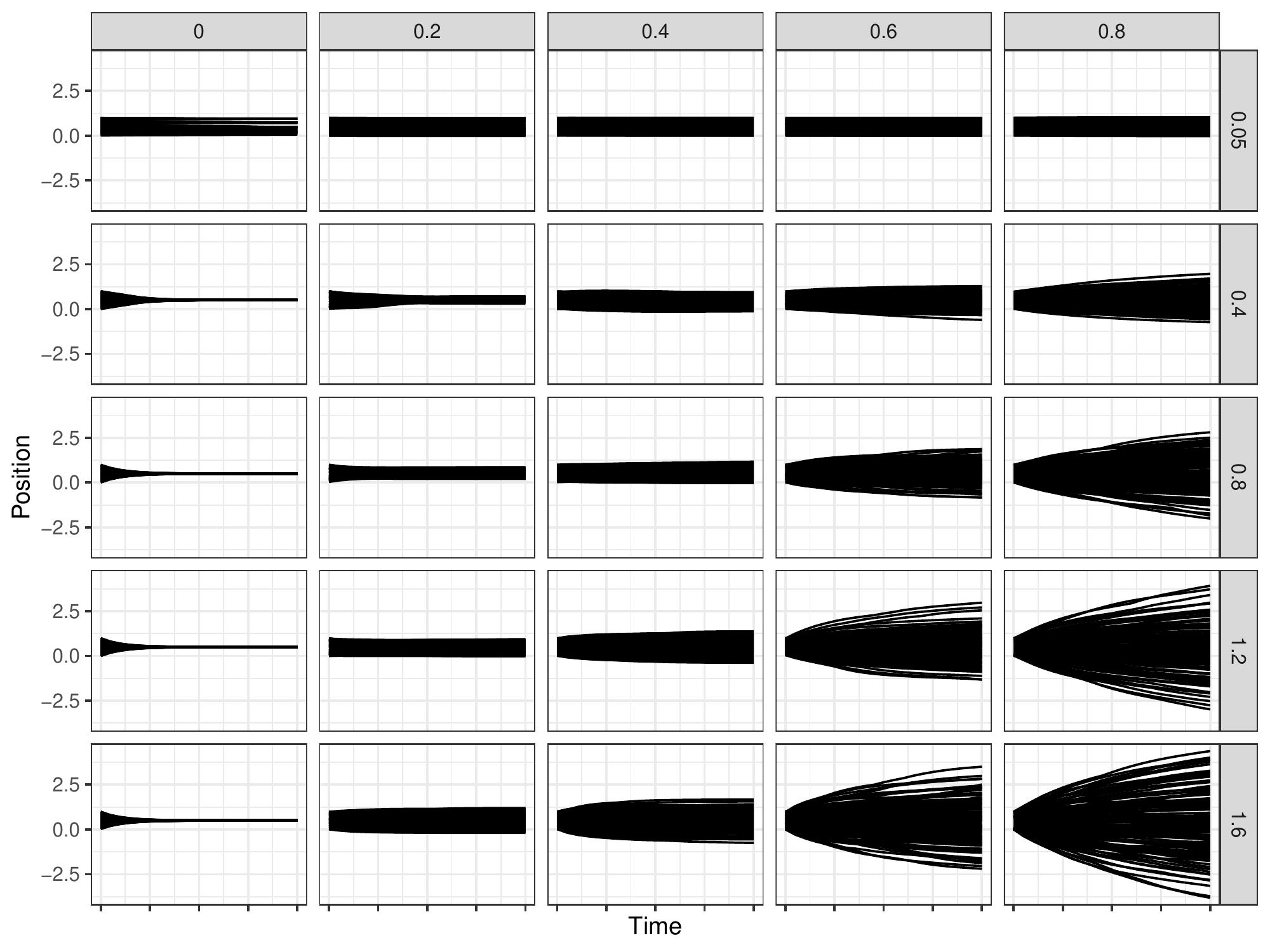}
    \caption{For all plots, $p_1 = .4$. The horizontal axis represents $p_2$, while the vertical axis represents $c$. Note that as $p_2$ increases, the range of final opinions gets wider. For low values of $p_2$, as $c$ increases, the range of final opinions becomes narrower (closer to consensus). By contrast, for high values of $p_2$, as $c$ increases, the range of final opinions becomes wider.}
    \label{fig:er}
\end{figure}

In particular, we observe that the proportion of $\frac{p_2}{p_1}$ seems to be the driving factor in the range of final opinions. To draw clearer conclusions, we look at {\em opinion spread}, which we define as the following quantity:

\begin{equation}
    \frac{\max_{i,j}|x_i(T) - x_j(T)|}{\max_{i,j}|x_i(0) - x_j(0)|}
\end{equation}

\begin{figure}[htbp!]
    \centering
    \includegraphics[width = \textwidth]{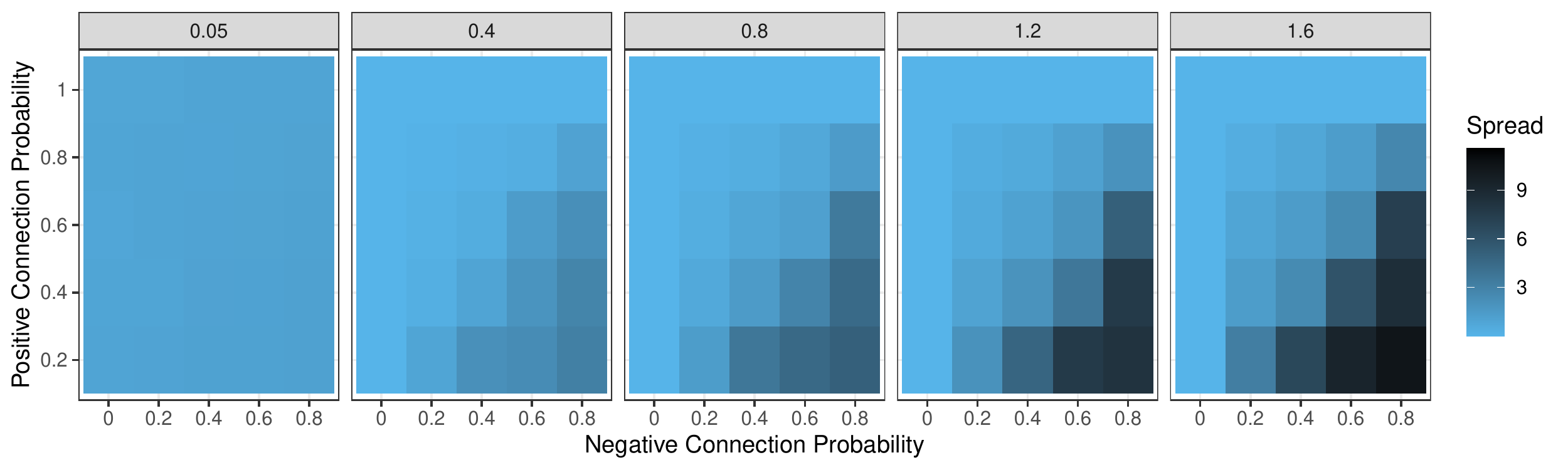}
    \caption{Heat map of opinion spread as a function of the probability of a connection in $G_1$ and $G_2$ iterated over confidence bound $c$. Data drawn from the mean of 100 trials for each set of parameters with parameters $p_1 \in \{0.2, 0.4, 0.6, 0.8, 1\}$, $p_2 \in \{0, 0.2, 0.4, 0.6, 0.8, 1\}$, and $c \in \{0.05, 0.4, 0.8, 1.2, 1.6\}$.}
    \label{fig:er_spread}
\end{figure}

In \cref{fig:er_spread}, we plot average opinion spread across 
trials as a function of the proportion $\frac{p_2}{p_1}$ and confidence bound $c$ in a heat map. We observe similar trends as in \cref{fig:er}, with higher proportions $\frac{p_2}{p_1}$ leading to higher values of opinion spread, and the influence of $c$ on opinion spread depending on $\frac{p_2}{p_1}$. In the following examples, we will similarly see that opinion spread is largely controlled by the negative edges in the network, but that the addition of more structure to the network will influence opinion formation in interesting ways.

\subsection{Stochastic Block Models}
Next, we adapt a Stochastic Block Model (SBM) in order to incorporate both positive and negative edges. In these networks, each node is assigned to a group $k\in K$. The probabilities of connections when $i_k = j_k$ is different than when $i_k \not= j_k$. This enforces structure within the network. Similarly to in \cref{sec:ER} we generate this network through two sub networks. In this case, the process begins with two SBM networks, $G_1 = G(n, p_1, \rho)$ and $G_2 = G(n,p_2, \rho)$, with associated adjacency matrices $A_1$ and $A_2$. The variable $p$ is the probability of having a connection with another node in the same cluster while $p\rho$ is the probability of having a connection with a node in a different cluster. If G is network represented by the adjacency matrix given by $A_1 - A_2$ we have the generated network edge probabilities:
\begin{align}
    P((i,j) \in E)_{i_k = i_j} &= \left\{
    \begin{array}{cc}
         0 & (1-p_1)(1-p_2) + p_1p_2 \\
         1 & p_1(1-p_2) \\
         -1 & (1-p_1)p_2
    \end{array}
    \right.\\
    P((i,j) \in E)_{i_k \not= i_j} &= \left\{
    \begin{array}{cc}
         0 & (1-p_1\rho)(1-p_2\rho) + \rho^2p_1p_2 \\
         1 & p_1\rho(1-p_2\rho) \\
         -1 & (1-p_1\rho)p_2\rho
    \end{array}
    \right.
    \label{eq:SBM}
\end{align}

A sample of the generative process can be seen in \cref{fig:Village_example}, where the blue edges represent positive edges and the black negative. Each color of nodes represents a group $k \in K$. 

\begin{figure}[!ht]
   \centering
     \begin{subfigure}[b]{0.3\textwidth}
         \centering
         \includegraphics[width=\textwidth]{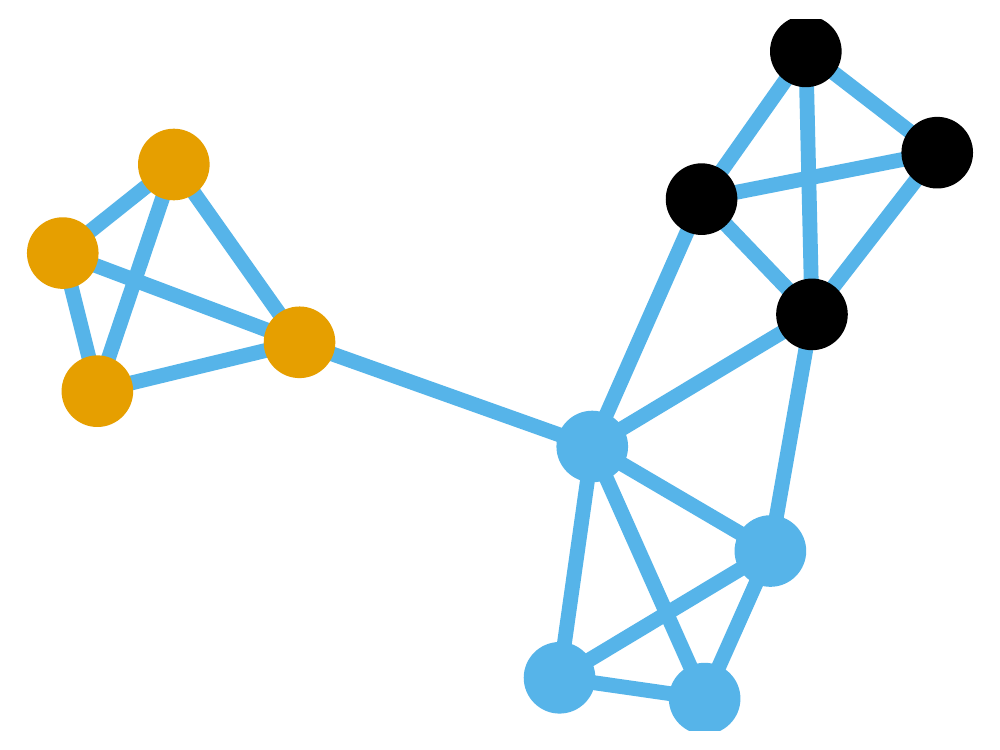}
         \caption{G1: positive node network with $p_1 = 0.85$ and $\rho = 0.05$}
         \label{fig:village_G1}
     \end{subfigure}
     \hfill
     \begin{subfigure}[b]{0.3\textwidth}
         \centering
         \includegraphics[width=\textwidth]{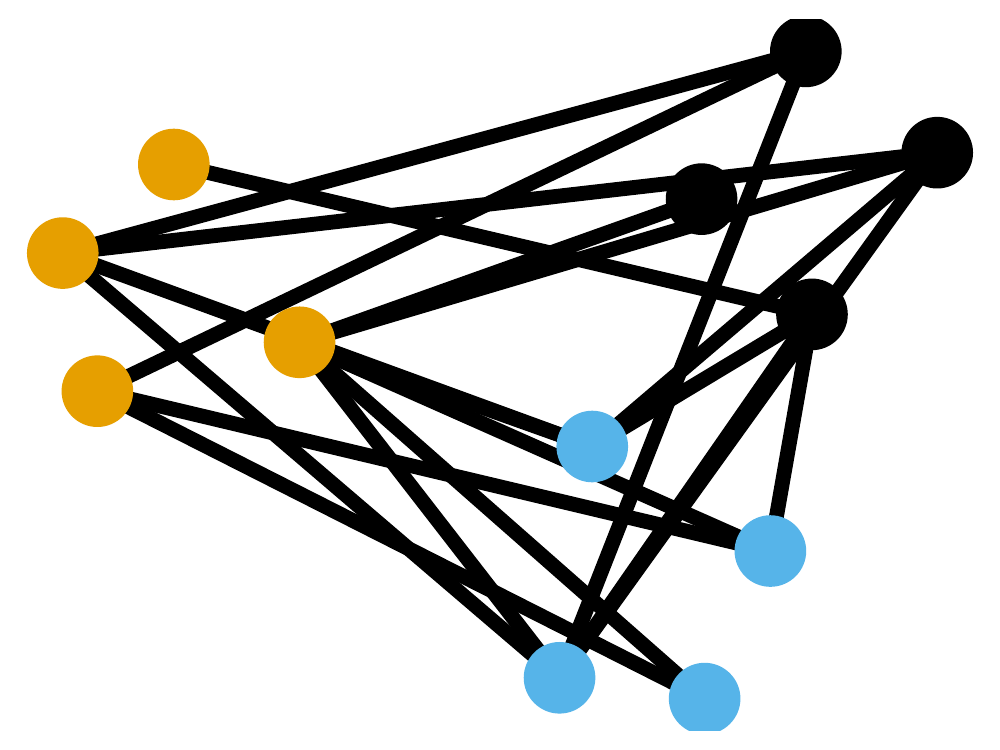}
         \caption{G2: negative node network with $p_2 = 0.3$ and $\rho = 0.05$ }
         \label{fig:village_G2}
     \end{subfigure}
      \hfill
     \begin{subfigure}[b]{0.3\textwidth}
         \centering
         \includegraphics[width=\textwidth]{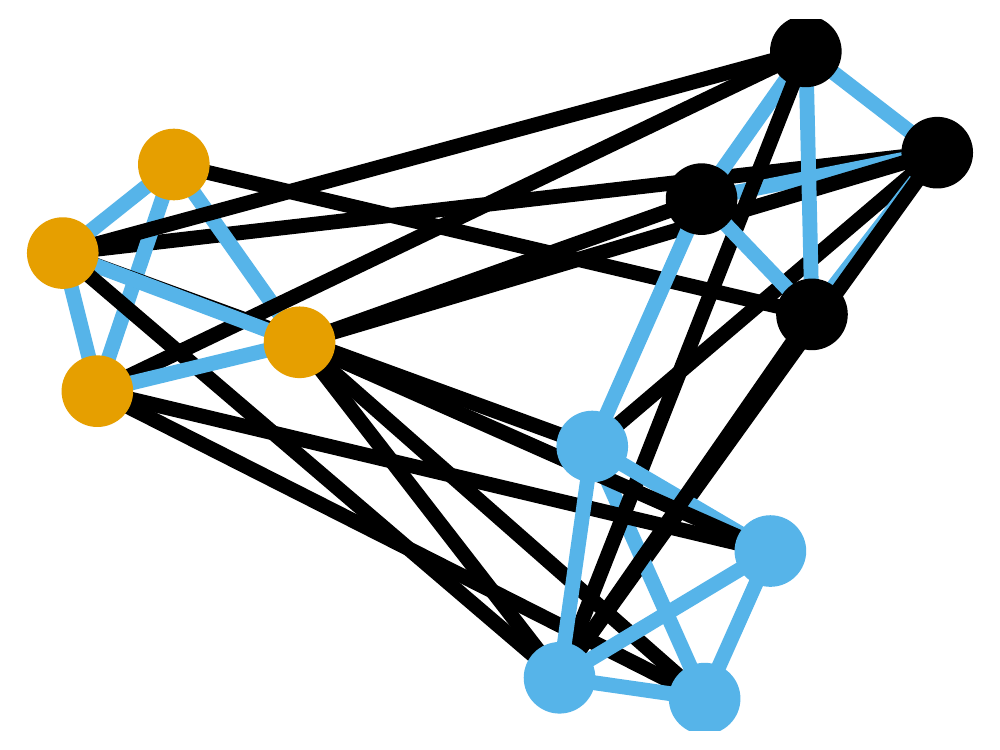}
         \caption{G3: final network $p_1 = 0.85$, $p_2 = 0.3$ and $\rho = 0.05$}
         \label{fig:village_G3}
     \end{subfigure}
     \caption{An example of the generation of the SBM network with attractive and repulsive edges}
     \label{fig:Village_example}
\end{figure}

Again, in order to create simulation results, 100 trials are run for all combinations of the parameters

\begin{align*}
\rho &\in (0, 0.2, 0.4, 0.6, 0.8, 1)\\
    p_1 &\in (0.2, 0.4, 0.6, 0.8, 1)\\
    p_2 &\in (0.0, 0.2, 0.4, 0.6, 0.8)\\
    c &\in (0.05, 0.4,  0.8,  1.2, 1.6)
\end{align*}

For each trial, a random set of initial opinions is generated and the model is applied. The full histories of an example run when $p_1 = 0.8$ and $p_2 = 0.2$ can be seen in \cref{fig:SBM_histories}. It can be seen that each group fully converges on itself, while the confidence interval determines how dispersed the groups are from each other. As $\rho$ increases, the full set begins to converge. As is the case with the ER models, for certain combinations the spread at steady-state is larger than that at the start. From these values we can see that when $p_1$ and $p_2$ are locked, higher confidence bounds result in a larger terminal spread, as does lower percentages of cross-group edges ($\rho$).

\begin{figure}[!ht]
    \centering
    \includegraphics[width = 0.8\textwidth]{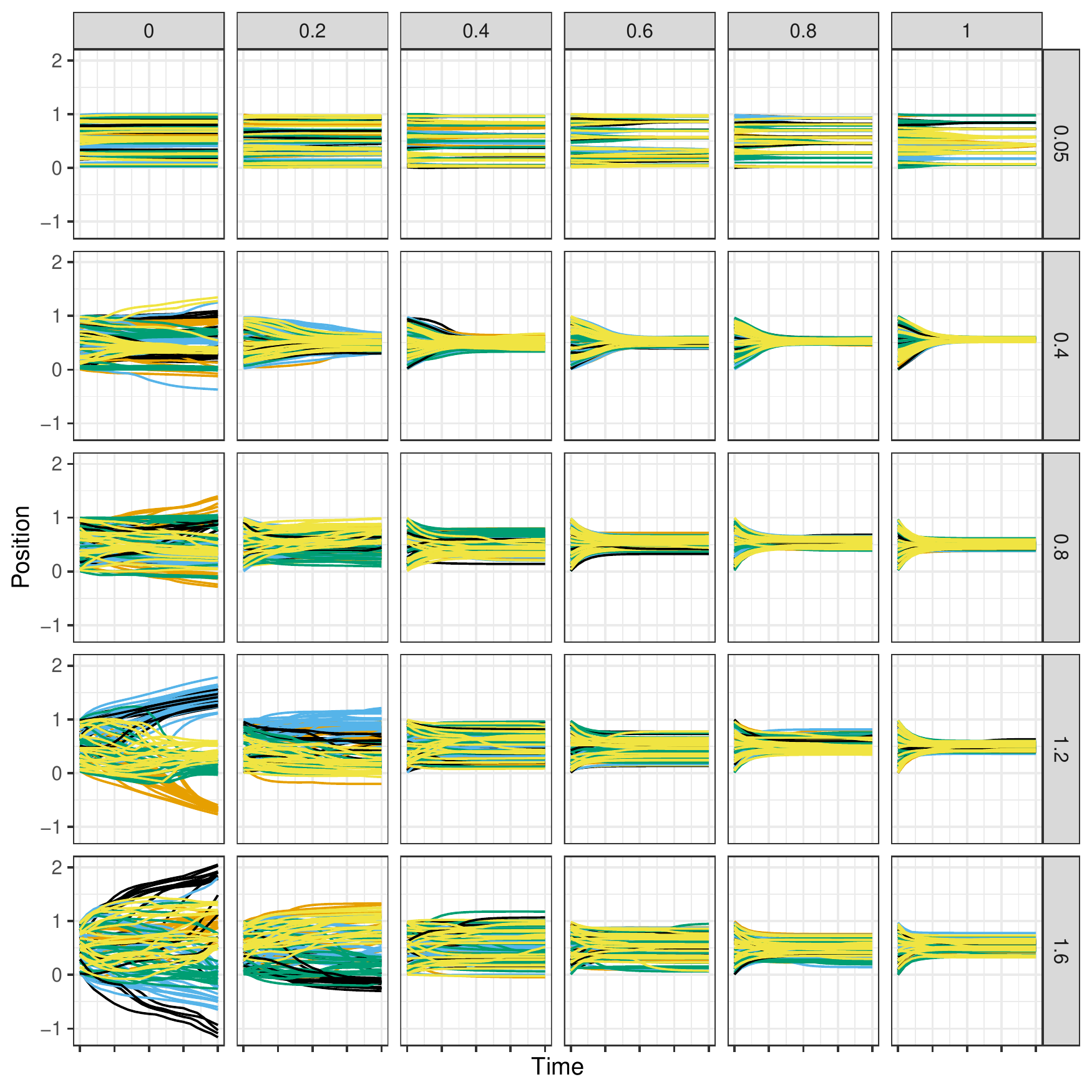}
    \caption{Example of the paths taken when the graph is distributed according to the Stochastic Block Model scheme laid out in \ref{eq:SBM}. In this case $p_1 = 0.8$ and $p_2 = 0.2$ are both locked the confidence bound varies between the rows and the percent of cross group links varies over columns. Each color represents a different group. There are five groups in these examples.}
    \label{fig:SBM_histories}
\end{figure}

As with the ER model, with the SBM we first look at steady-state opinion spread. The results can be seen in \cref{fig:village_opinion_spread}. Note, that when $\rho = 1$ the SBM model is equivalent to the ER model for the same parameters. We therefore have that the final row is identical to \cref{fig:er_spread}. We note that as the value of $\rho$ shrinks, the final spread increases. Otherwise, the trends found for the ER model are consistent.

\begin{figure}[!ht]
    \centering
    \includegraphics[width = \textwidth]{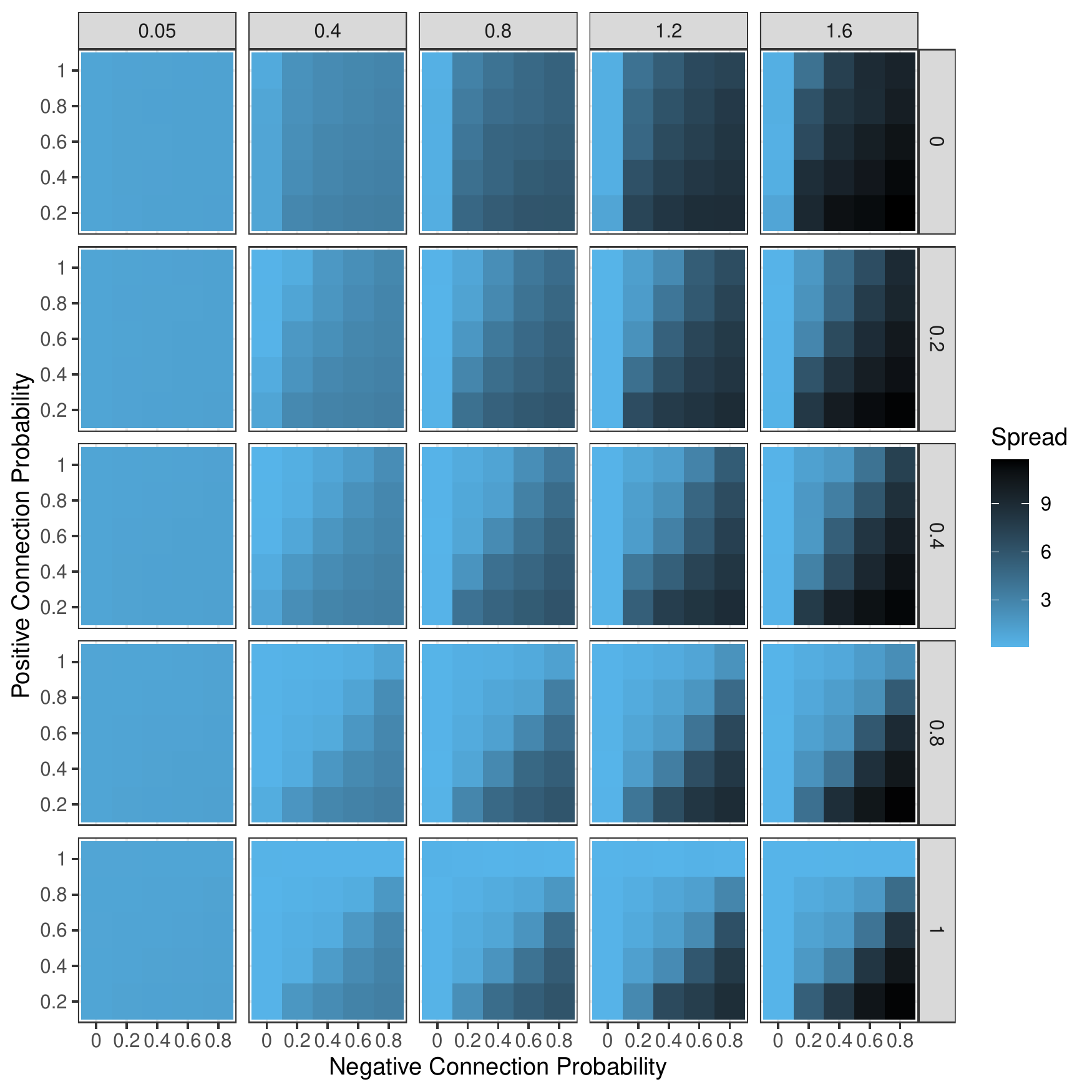}
    \caption{Heat map of the opinion spread of Stochastic Block Model trials. The rows represent $\rho$, the percent of positive(negative) connection values that are out (in) group, while the columns are $c$ the confidence bound. }
    \label{fig:village_opinion_spread}
\end{figure}

In the case of SBMs, since each vertex $i \in V$ is assigned to a group in $k$, we are also interested in clustering in addition to opinion spread. We introduce a measures of how close vertices are to in-group vertices versus out-group vertices. In order to calculate this, which we call \textit{proportional spread}, first we find the average in and out-group distances as:

\begin{equation}
    \mathcal{I}_T =  \frac{1}{|k|}\sum_{k_i \in k} \frac{1}{|k_i|^2}\sum_{j \in k_i \subset V} \sum_{\ell \in k_i \in \subset V} |x_{j}(T) - x_{\ell}(T)|
\end{equation}
\begin{equation}
    \mathcal{O}_T = \frac{1}{|k|}\sum_{k_i \in k} \frac{1}{|k_i|}\sum_{j \in k_i \subset V} \frac{1}{|k| - |k_i| }\sum_{\ell \in \subset V/k_i} |x_{j}(T) - x_{\ell}(T)|
\end{equation}

Since the end-spread of the samples differ, in order to appropriately compare them we look at the ratio, that is

\begin{equation}
    \mathcal{PS}_T = \frac{\mathcal{I}_T}{\mathcal{O}_T}
\end{equation}

Smaller values of proportional spread imply same-group nodes are significantly closer to each other than different-group nodes. Thus, small values imply increased separation by group. In \cref{fig:start_hist} the spread as well as $\mathcal{O}_{0}$ and $\mathcal{I}_{0}$ can be seen. It is clear that the values range most of the space, in addition $\mathcal{O}_{0} \approx \mathcal{I}_{0}$. Thus, uniformly, at the beginning of the trials we have $\mathcal{PS}_0 =  1$. The plots of proportional spread at steady-state over trials can be seen in \cref{fig:PS}. We observe that clustering is most clear with higher values of $p_1$ and lower values of $\rho$. To a lesser extent, clustering also increases with lower values of $p_2$. This implies that the higher rates of in-group connection and fewer cross group connections lead to increased polarization.

\begin{figure}[!ht]
    \centering
    \includegraphics[width = 0.8\textwidth]{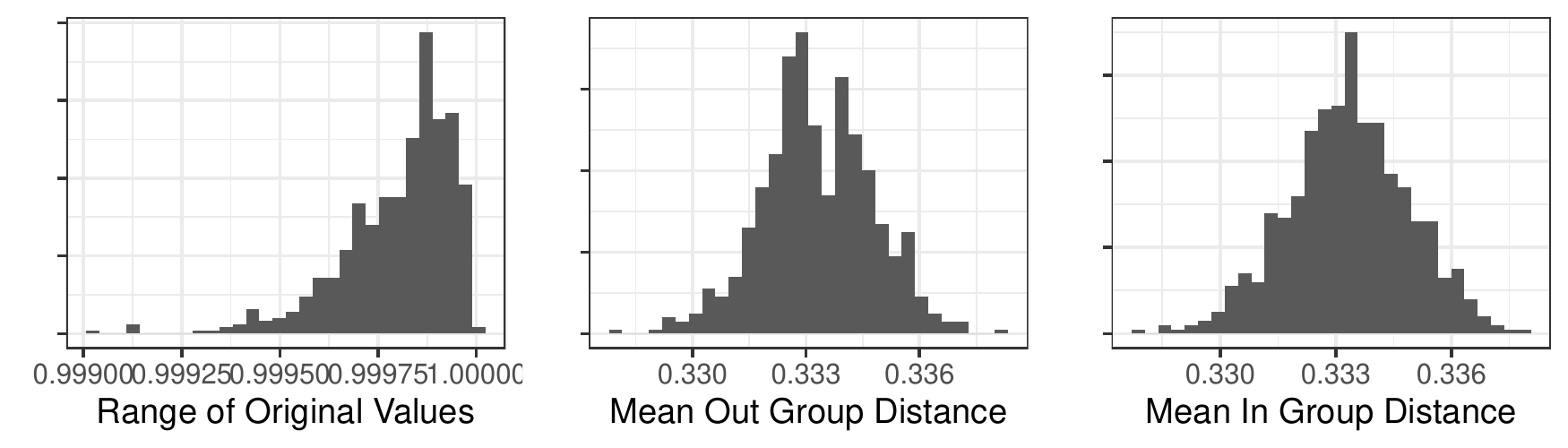}
    \caption{Histograms of metrics at the the start of each trial, useful for comparing the final results}
    \label{fig:start_hist}
\end{figure}

\begin{figure}[!ht]
    \centering
    \includegraphics[width = 0.8\textwidth]{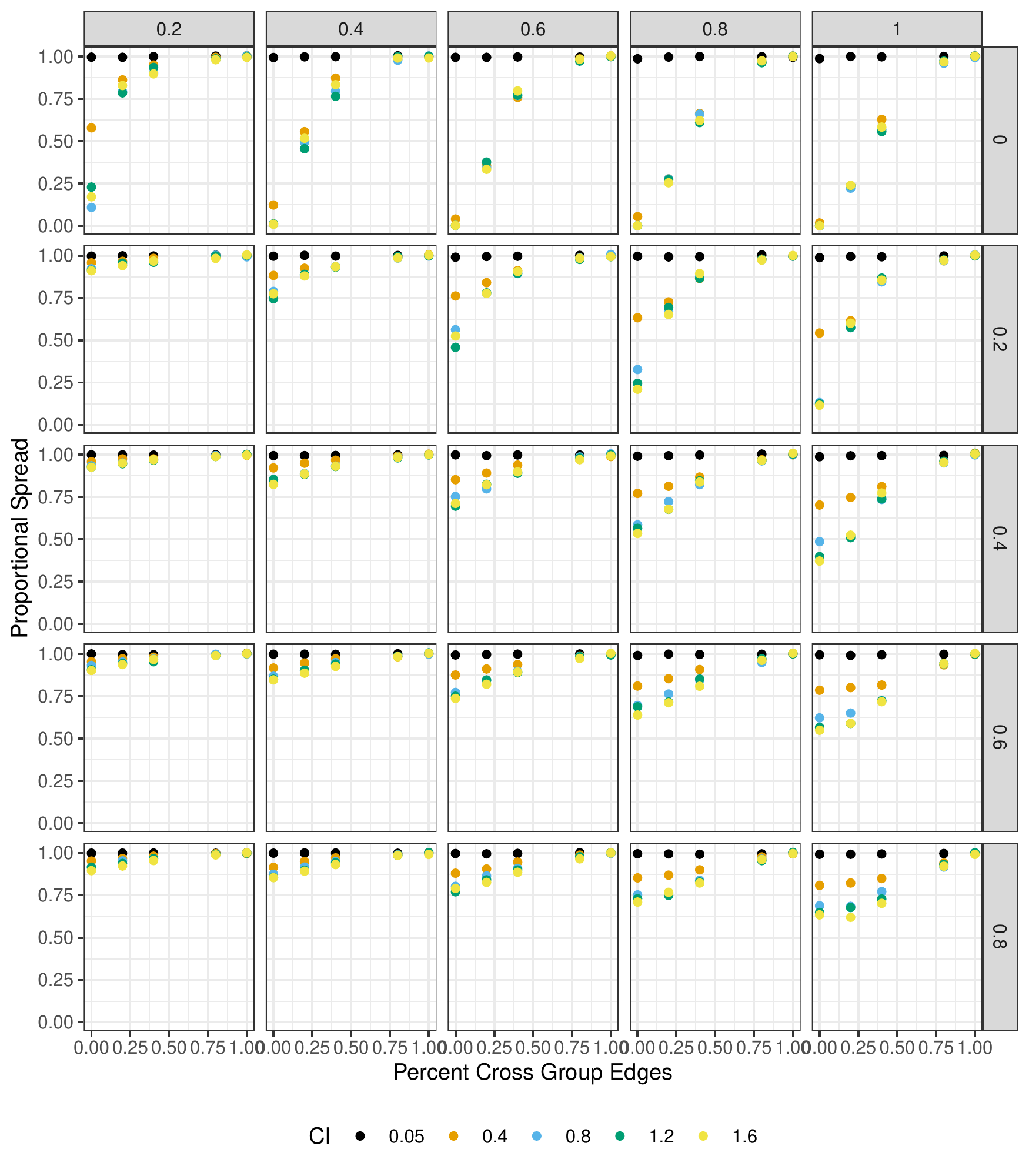}
    \caption{Proportional spread for Stochastic Block Model trials. The rows represent $p_2$, the base probability of an out-group negative connection, while the columns are $p_1$ the base in-group positive probability. The x-axis is $\rho$ the percent of cross group edges and the color represents the confidence bound. A value of 1 means that nodes are equally as close to in and out-group nodes. Smaller values mean groups are increasingly clustered.}
    \label{fig:PS}
\end{figure}





\clearpage
\section{Future Work and Conclusions}
\label{sec:conclusions}
Previous models of opinion dynamics have focused on the attractive nature of network connections. These result in a complete convergence in the steady-state for each receptivity subgraph. When the term \textit{polarization} is used, it still implies a shrinking of the overall opinion space. The model introduced in this paper, in contrast, acknowledges that there are circumstances in which individuals seek to differentiate themselves from those in their network. This behavior leads to the possibility of an overall opinion space expansion.

In this paper, the basic bounds of this expansion were proven analytically in certain cases, with intuition provided for the general case. In addition, the steady-state behavior for random networks were analyzed numerically. These simulation results offered insight into the effects of various structural parameters on the model. It was clear that when there was structure in the network links (for instance in the Stochastic Block Model) group-based clustering emerged. This clustering was despite the random initial conditions provided. In addition, the size of the confidence bounds and the density of repulsive edges are both pivotal in the opinion spread.

The model introduced in this paper lends itself to complex opinion dynamics, where politicians or individuals want to differentiate themselves due to factors orthogonal to their expressed opinions. In future work we hope to explore how this model can help us understand political behavior. Initial ideas include using informative initial conditions for congressional networks. Alternatively, this model can be used to look at ideological opinions of individuals who are influenced by pop culture associating ideological beliefs with other factors. This would introduce a variable connection to an ideal point, which then attracts or repulses the individual. The addition of repulsive forces make bounded-confidence models increasingly relevant for empirical and theoretical studies of human behavior.

\section*{Acknowledgments}
One of the authors (M. Feng) on this project is funded by the James S. McDonnell Foundation Postdoctoral Fellowship. In addition, we would like to thank Danny Ebanks, R. Michael Alvarez, Jonathan Katz, and Mason A. Porter for helpful comments and insights. 

\bibliographystyle{siamplain}
\bibliography{references}

\end{document}


\maketitle

\section{A detailed example}

Here we include some equations and theorem-like environments to show
how these are labeled in a supplement and can be referenced from the
main text.
Consider the following equation:
\begin{equation}
  \label{eq:suppa}
  a^2 + b^2 = c^2.
\end{equation}
You can also reference equations such as \cref{eq:matrices,eq:bb} 
from the main article in this supplement.

\lipsum[100-101]

\begin{theorem}
  An example theorem.
\end{theorem}

\lipsum[102]
 
\begin{lemma}
  An example lemma.
\end{lemma}

\lipsum[103-105]

Here is an example citation: \cite{KoMa14}.

\section[Proof of Thm]{Proof of \cref{thm:bigthm}}
\label{sec:proof}

\lipsum[106-112]

\section{Additional experimental results}
\Cref{tab:foo} shows additional
supporting evidence. 

\begin{table}[htbp]
{\footnotesize
  \caption{Example table}  \label{tab:foo}
\begin{center}
  \begin{tabular}{|c|c|c|} \hline
   Species & \bf Mean & \bf Std.~Dev. \\ \hline
    1 & 3.4 & 1.2 \\
    2 & 5.4 & 0.6 \\ \hline
  \end{tabular}
\end{center}
}
\end{table}

\bibliographystyle{siamplain}
\bibliography{references}